\documentclass{amsart}
\usepackage{amsmath}
\usepackage{amsfonts}
\usepackage{amssymb}
\usepackage[utf8]{inputenc}
\usepackage{amssymb}
\usepackage{mathrsfs}
\usepackage{amsthm}

\usepackage{enumerate}

\usepackage{graphicx}

\newtheorem{thrm}{Theorem}[section]

\newtheorem{lemma}[thrm]{Lemma}
\newtheorem{prop}[thrm]{Proposition}

\theoremstyle{definition}

\newtheorem{rem}[thrm]{Remark}

\DeclareMathOperator{\cl}{cl}

\begin{document}

\title[Some remarks on the projective properties]{Some remarks on the projective properties of Menger and Hurewicz}
\author[M.\ Krupski]{Miko\l aj Krupski}
\address{Departamento de Matem\'{a}ticas \\ Universidad de Murcia \\ Campus de Espinardo \\ 30100 Murcia \\ Spain \\ and \\ Institute of Mathematics\\ University of Warsaw\\ ul. Banacha 2\\
02--097 Warszawa, Poland }
\email{mkrupski@mimuw.edu.pl}
\author[K.\ Kucharski]{Kacper Kucharski}
\address{Institute of Mathematics\\ University of Warsaw\\ ul. Banacha 2\\
02--097 Warszawa, Poland }
\email{k.kucharski6@uw.edu.pl}

\begin{abstract}
It is known that both the Menger and Hurewicz property of a Tychonoff space $X$ can be described by the way $X$ is placed in its \v{C}ech-Stone compactification $\beta X$. We provide analogous characterizations for the projective versions of the properties of Menger and Hurewicz.
\end{abstract}

\subjclass[2020]{Primary: 54D20, 54D40, 91A44}

\keywords{Menger space, Hurewicz space, projective Menger property, projective Hurewicz property, Porada game}

\maketitle

\section{Introduction}
All spaces in this note are assumed to be Tychonoff topological spaces. Let $\mathcal{P}$ be a topological property. We say that a space $X$ is \textit{projectively $\mathcal{P}$} provided every separable metrizable continuous image of $X$ has property $\mathcal{P}$. This notion was studied by several authors for different topological properties $\mathcal{P}$. In this note we are concerned with the case when $\mathcal{P}$ is either the Menger property or the property of Hurewicz.

Let us recall that a topological space $X$ is \textit{Menger} (resp., \textit{Hurewicz})
if for every sequence $(\mathscr{U}_n)_{n\in \omega}$ of open covers of $X$, there is a sequence $(\mathscr{V}_n)_{n\in\omega}$
such that for every $n$, $\mathscr{V}_n$ is a finite subfamily of $\mathscr{U}_n$ and the family $\bigcup_{n\in\omega} \mathscr{V}_n$ covers $X$
(resp., every point of $X$ belongs to $\bigcup\mathscr{V}_n$ for all but finitely many $n$'s).

The projective versions of the Menger and Hurewicz property were a subject of study before; see, e.g., Ko\v{c}inac \cite{Ko} or Bonanzinga, Cammaroto and Matveev \cite{BCM}. What makes them useful is the following fact (see \cite[Proposition 2]{T}, \cite[Proposition 8]{BCM}, \cite[Theorem 3.2]{Ko}, \cite[Proposition 31]{BCM}):
\begin{prop}
A space $X$ is Menger (resp., Hurewicz) if and only if $X$ is Lindel\"of and projectively Menger (resp., projectively Hurewicz).
\end{prop}

It is known that both the Menger and Hurewicz property of a space $X$ can be conveniently characterized by the way $X$ is positioned in its \v{C}ech-Stone compactification $\beta X$. The aim of the present note is to provide analogous descriptions for their projective versions. Our interest in this sort of characterizations
primarily stems from a recent work by the first author \cite{Kr1}, \cite{Kr2}, where they play an important role.


\section{Notation}

In what follows, the set of all natural numbers (including $0$) will be denoted by $\omega$. By $\omega^{< \omega} = \bigcup_{n \in \omega} \omega^n$ we will mean the set of all finite sequences with values in $\omega$ (including $\emptyset$), and as a natural extension, by $\omega^{\omega}$ we will denote the set of all infinite sequences with values in $\omega$.
If $s=(a_0,\ldots ,a_{n-1})\in \omega^{<\omega}$ and if $k\leq n$, then $s|k=(a_0,\ldots a_{k-1})$
is an initial segment of $s$ of length $k$. By $s|0$ we mean the empty sequence. The symbol $|s|$ stands for the length of $s$. 
For $s=(a_0,\ldots ,a_{k})\in \omega^{<\omega}$ and $n\in\omega$, by $s^{\frown}n$ we will denote the sequence $(a_0,\ldots , a_{k},n)$.

For a space $X$ by $\beta X$ we denote the \v{C}ech-Stone compactification of $X$. Recall that a subset $A$ of a topological space $X$ is called a \textit{zero-set} if there is a continuous map $f:X\to [0,1]$ such that $A=f^{-1}(0)$. According to Vedenissov's lemma (see \cite[1.5.12]{Eng}) if $Z$ is a compact space, then $A$ is a zero-set in $Z$ if and only if $A$ is closed $G_\delta$-subset of $Z$. The complement of a zero-set is called a \textit{cozero-set}.

Let us describe two topological games that will be of interest: the $k$-Porada game introduced by Telg\'{a}rsky in \cite{T}, and our own modification of it which we shall call the $z$-Porada game.

Let $Z$ be a compact space and let $X\subseteq Z$ be a subspace of $Z$.
The \textit{$k$-Porada game} on $Z$ with values in $X$ is a game with $\omega$-many innings, played alternately
by two players: I and II.
Player I begins the game and makes the first move by choosing
a pair $(K_0,U_0)$, where $K_0\subseteq X$ is a nonempty compact set and $U_0$ is an open set in $Z$ that contains $K_0$. Player II responds by choosing an open
(in $Z$) set $V_0$ such that
$K_0\subseteq V_0\subseteq U_0$. In the second round of the game, player I picks a pair $(K_1,U_1)$, where
$K_1 \subseteq X$ is a nonempty compact subset of $V_0$ and $U_1$ is an open subset of $Z$ with
$K_1\subseteq U_1\subseteq V_0$. Player II responds by picking an open (in $Z$) set $V_1$ such that $K_1\subseteq V_1\subseteq U_1$. The game continues in this way and stops
after $\omega$ many rounds. Player II wins the game if $\emptyset\neq \bigcap_{n\in\omega} U_n (= \bigcap_{n\in \omega} V_n) \subseteq X$.
Otherwise player I wins.
The game described above is denoted by $kP(Z,X)$.

Let us introduce the following modification of the $k$-Porada game. As above, $Z$ is a compact space and $X$ is a subspace of $Z$. The \textit{$z$-Porada game} on $Z$ with values in $X$, denoted by $zP(Z,X)$, is played as $kP(Z,X)$ with the only difference that compact sets $K_n$ played by player I are required to be zero-sets in $Z$ (i.e. compact $G_\delta$). We keep the requirement that these sets are contained in $X$.

For a space $Y$ we denote by $\mathcal{T}_Y$ (resp., $\mathcal{K}(Y)$) the collection of all nonempty open (resp., compact) subsets of $Y$. For a space $Y$ and its subspace $X$ we denote by $\mathcal{Z}(Y,X)$ the collection of all nonempty zero-sets in $Y$ contained in $X$.
A \textit{strategy} for player I in the game $kP(Z,X)$ (resp., in the game $zP(Z,X)$) is a map $\sigma$ defined inductively as follows: Set $V_{-1}=Z$.
If the strategy $\sigma$ is defined for the first $n$ moves, $n\in\omega$, then an $n$-tuple $(V_0,V_1,\ldots,V_{n-1})\in \mathcal{T}_Z^n$ is called
\textit{admissible} if either the tuple is empty (i.e. $n=0$) or else $K_0\subseteq V_0\subseteq U_0$ and both $K_i\subseteq V_i\subseteq U_i$, and $(K_i,U_i)=\sigma(V_0,\ldots V_{i-1})$ for $i\in\{1,\ldots , n-1\}$.
For any
admissible $n$-tuple $(V_0,\ldots, V_{n-1})$ we choose a pair $(K_n,U_n)\in \mathcal{K}(V_{n-1}\cap X)\times \mathcal{T}_{V_{n-1}}$ (resp., $(K_n,U_n)\in \mathcal{Z}(Z,V_{n-1}\cap X)\times \mathcal{T}_{V_{n-1}}$) with $K_n\subseteq U_n$ and we set
$$\sigma(V_0,\ldots ,V_{n-1})=(K_n,U_n).$$

A strategy $\sigma$ for player I in either of the games $kP(Z,X)$ or $zP(Z,X)$, is called \textit{winning} if player I wins every run of the game in which she plays according to the strategy $\sigma$.

If $f: X \rightarrow Y$ is a function, then for $A\subseteq X$ we set
$$f^{\#}(A) = Y \setminus f(X \setminus A).$$

The following lemma is immediate.

\begin{lemma}\label{obraz_kreteczka}
Suppose that $f:X \rightarrow Y$ is a continuous surjection between spaces $X$ and $Y$. We have:
\begin{enumerate}[(i)]
    \item If $X$ is compact and $U \subseteq X$ is open, then $f^{\#}(U)$ is open in $Y$
    \item For any $A\subseteq X$ and $y\in Y$ we have $y\in f^{\#}(A)$ if and only if $f^{-1}(y)\subseteq A$.
    \item For any $A\subseteq X$ and $B\subseteq Y$, if $A\subseteq f^{-1}(B)$, then $f^{\#}(A)\subseteq B$.
\end{enumerate}
\end{lemma}

We will use a standard notation for the closure operator, i.e. if $X$ is a space and $A\subseteq X$, then the closure of $A$ in $X$ will be denoted by $\cl_X(A)$.

\section{The projective Hurewicz property}

The following characterization of the Hurewicz property was established by Just, Miller, Scheepers and Szeptycki \cite{JMSS} (for the subsets of the real line), Banakh and Zdomskyy \cite{BZ} (for separable metrizable spaces) and Tall \cite{Tall} (the general case).

\begin{thrm}\label{Hurewicz characterization}
Let $bX$ be a compactification of $X$. The following two conditions are equivalent:
\begin{enumerate}[(A)]
    \item $X$ has the Hurewicz property
    \item For every $\sigma$-compact subset $F$ of the remainder $b X\setminus X$,
    there exists a $G_\delta$-subset $G$ of $b X$ such that $F\subseteq G\subseteq bX\setminus X$.
\end{enumerate}
\end{thrm}

The following result was obtained by Bonanzinga, Cammaroto and Matveev (see \cite[Theorem 30]{BCM}).

\begin{thrm}(Bonanzinga, Cammaroto, Matveev) The following conditions are equivalent:
\begin{enumerate}[(A)]
    \item X is projectively Hurewicz,
    \item For every sequence $(\mathscr{U}_n)_{n\in \omega}$ of countable covers of $X$ by cozero-sets, there is a sequence $(\mathscr{V}_n)_{n\in\omega}$ such that for every $n$, $\mathscr{V}_n$ is a finite subfamily of $\mathscr{U}_n$ and for all $x \in X$, point $x$ belongs to $\bigcup \mathscr{V}_n$, for all but finitely many $n$'s (i.e. the family $\{\bigcup \mathscr{V}_n \; : \; n\in\omega\}$ is a $\gamma$-cover of $X$).
\end{enumerate}
\end{thrm}

The above theorem suggests the following counterpart of Theorem \ref{Hurewicz characterization}.

\begin{thrm}\label{projective Hurewicz characterisation}
For a space $X$, the following conditions are equivalent:
\begin{enumerate}[(A)]
    \item $X$ is projectively Hurewicz,
    \item For any set $F\subseteq \beta X\setminus X$ which is a countable union of zero-sets in $\beta X$, there exists a $G_\delta$-subset $G$ of $\beta X$ such that $F \subseteq G \subseteq \beta X \setminus X$.
\end{enumerate}
\end{thrm}

\begin{proof}
Let us assume that $X$ is projectively Hurewicz and let $(F_n)_{n \in \omega}$ be a sequence of zero-sets in $\beta X$ such that
\[
F = \bigcup_{n \in \omega} F_n \subseteq \beta X \setminus X.
\]
For $n\in \omega$ let
$f_n : \beta X \to [0,1]$
be a continuous map witnessing that $F_n$ is a zero-set, i.e.
$F_n = f_n^{-1}(0)$. Let
\[
f: \beta X \to [0,1]^{\omega},
\]
be the diagonal map given by the family $\{f_n:n\in \omega\}$, i.e.
$$f(x) = (f_n(x))_{n<\omega}.$$

Denote $f(X) = Y \subseteq [0,1]^{\omega}$ and $f(\beta X) = b Y$. 
Observe that
\begin{equation}\label{equation 1}
    \bigcup_{n\in \omega} f(F_n)\subseteq bY\setminus Y.
\end{equation}

Indeed, if $y\in f(F_n)$, then $y(n)=0$ so $f^{-1}(y)\subseteq f_n^{-1}(0)=F_n\subseteq \beta X\setminus X$. On the other hand $y\in Y$ would give $f^{-1}(y)\cap X\neq\emptyset$, because $f$ maps $X$ onto $Y$.

By our assumption, $X$ is projectively Hurewicz, so $Y$ has the Hurewicz property being a separable metrizable continuous image of $X$. It follows from \eqref{equation 1} and Theorem \ref{Hurewicz characterization} that there is
a $G_\delta$-subset $H$ of $bY$ with
\[
\bigcup_{n\in\omega} f(F_n) \subseteq H \subseteq bY \setminus Y.
\]
We set $G=f^{-1}(H)$.
Since $f$ maps $X$ onto $Y$, we get
$F\subseteq G\subseteq \beta X\setminus X.$


Conversely, assume condition $(B)$. We need to show that $X$ is projectively Hurewicz.
Let $Y$ be a separable metric space and let $f:X \to Y$ be a continuous surjection. Let $bY$ be a metrizable compactification of $Y$.

The map $f$ can be uniquely extended to the continuous map $\widetilde{f}:\beta X \rightarrow bY$.
In order to prove that $Y$ has the Hurewicz property, we will apply Theorem \ref{Hurewicz characterization}. To this end, let us fix a $\sigma$-compact subset $F$ of $bY \setminus Y$. We need to find a $G_\delta$-set $H$ in $bY$, such that $F \subseteq H \subseteq bY \setminus Y$.

Let $D=\widetilde{f}^{-1}(F)$. Since $F$ is a $\sigma$-compact subset of a metric space $bY$, it is a countable union of zero-sets so the set $D$ is a countable union of zero-sets too. Moreover, since $\widetilde{f}$ maps $X$ onto $Y$ and $F\cap Y=\emptyset$, we have $D\subseteq \beta X\setminus X$. Now, condition $(B)$ provides a $G_\delta$-set $G$ in $\beta X$ such that
$$D \subseteq G \subseteq \beta X \setminus X.$$

Write $G=\bigcap_{n\in \omega} G_n$, where $G_n$ is open for every $n\in \omega$. For each $n\in \omega$ we set
$$H_n=\widetilde{f}^{\#}(G_n).$$
By Lemma \ref{obraz_kreteczka} the set $H_n$ is open, for every $n\in \omega$.
We claim that $H=\bigcap_{n \in \omega} H_n$ is the desired $G_\delta$-subset of $bY$, i.e.
$$F \subseteq \bigcap_{n \in \omega} H_n \subseteq bY \setminus Y.$$
Indeed, if $y\in F$, then $\widetilde{f}^{-1}(y)\subseteq D\subseteq G$, so $\widetilde{f}^{-1}(y)\subseteq G_n$, for every $n\in \omega$. This gives $y\in H_n$, for every $n\in \omega$ (cf. Lemma \ref{obraz_kreteczka}). For the second inclusion, fix $y\in \bigcap_{n\in \omega}H_n$. It follows from Lemma \ref{obraz_kreteczka} that $\widetilde{f}^{-1}(y)\subseteq \bigcap_{n\in \omega}G_n=G$. Since $G$ is disjoint from $X$ and $\widetilde{f}$ maps $X$ onto $Y$  we have $y\in bY\setminus Y$.
\end{proof}


\section{The projective Menger property}

The theorem below, characterizing the Menger property, is due to Telg\'arsky \cite[Theorem 2]{T}.

\begin{thrm}(Telg\'arsky)\label{thrm_Telgarsky}
Let $bX$ be a compactification of $X$. The following two conditions are equivalent.
\begin{enumerate}[(A)]
    \item $X$ has the Menger property
    \item Player I has no winning strategy in the $k$-Porada game\\ $kP(b X,b X\setminus X)$
\end{enumerate}
\end{thrm}

The following result was established by Bonanzinga, Cammaroto and Matveev (see \cite[Theorem 6]{BCM}).

\begin{thrm}(Bonanzinga, Cammaroto, Matveev) The following conditions are equivalent:
\begin{enumerate}[(A)]
\item $X$ is projectively Menger,
\item For every sequence $(\mathscr{U}_n)_{n\in \omega}$ of countable covers of $X$ by cozero-sets, there is a sequence $(\mathscr{V}_n)_{n\in\omega}$
such that for every $n$, $\mathscr{V}_n$ is a finite subfamily of $\mathscr{U}_n$ and the family $\bigcup_{n\in\omega} \mathscr{V}_n$ covers $X$.
\end{enumerate}

\end{thrm}

This suggests the following counterpart of Theorem \ref{thrm_Telgarsky}.

\begin{thrm}\label{projectively_Menger_characterisation}
The following two conditions are equivalent:
\begin{enumerate}[(A)]
    \item $X$ has the projective Menger property
    \item Player I has no winning strategy in the $z$-Porada game\\ $zP(\beta X,\beta X\setminus X)$
\end{enumerate}
\end{thrm}

\begin{proof}
Assume that $X$ is projectively Menger. Let $\sigma$ be a strategy for player I in the $z$-Porada game $zP(\beta X,\beta X\setminus X)$. We need to show that there is a play
$$\sigma(\emptyset),\; V_0,\; \sigma(V_0),\; V_1,\;\sigma(V_0,V_1),\;\ldots$$
where player I applies her strategy and fails, i.e. $\emptyset\neq\bigcap_{n\in \omega}{V_n}\subseteq \beta X\setminus X$.


By induction on $|s|$, for $s\in \omega^{<\omega}$, we construct
an open subset $V_s$ of $\beta X$ and 
a zero-set $K_s$ in $\beta X\setminus X$
such that:
\begin{enumerate}[(i)]
\item For every $s\in \omega^{<\omega}$, the tuple $\left(V_{s|0}, V_{s|1}, \ldots , V_{s}\right)$ is admissible.
\item If $\sigma(V_{s|0}, V_{s|1}, \ldots , V_{s})=(K,U)$, then $K_{s}=K$, for every $s\in \omega^{<\omega}$.
\item $\cl_{\beta X}(V_{s\frown (n+1)})\subseteq V_{s\frown n}$, for every $s\in \omega^{<\omega}$ and $n\in \omega$
\item $K_s=\bigcap_{n\in \omega} V_{s\frown n}$ for every $s\in \omega^{< \omega}$
\item $\cl_{\beta X}(V_{s\frown n})\subseteq V_{s}$, for every $s\in \omega^{< \omega}$ and $n\in \omega$.
\end{enumerate}

\medskip

We have $\sigma(\emptyset)=(L,W)$, for some $L\in \mathcal{Z}(\beta X,\beta X\setminus X)$ and an open subset $W$ of $\beta X$ satisfying $L\subseteq W$.

Set $V_\emptyset = W$. Fix $m\in \omega$ and let $p\in \omega^{<\omega}$ be a sequence of length $m$. Suppose that the sets $V_s$ are constructed for all $s$ satisfying $|s|\leq m$ and the sets $K_s$ are constructed for all $s$ satisfying $|s|<m$, in such a way that the conditions (i)--(v) are satisfied. We will define $K_p$ and $V_{p\frown n}$, for all $n\in \omega$.
Since (i) holds for $p$, the pair
$$\sigma\left(V_{p|0}, V_{p|1}, \ldots , V_{p}\right)=(K,U)$$
is well defined and $K\subseteq U\subseteq V_p$. We set $K_p=K$. Now, the set $K_p$ is a zero-subset of $\beta X$ so we can write
\begin{align*}
K_p=\bigcap_{n\in \omega} V'_n,&\mbox{ where the sets } V'_n \mbox{ are open and }\\
&\cl_{\beta X}(V'_{n+1})\subseteq V'_n \subseteq  \cl_{\beta X}(V'_{n})\subseteq U, \mbox{ for all }n\in\omega.
\end{align*}
We set $V_{p\frown n}=V'_n$. This finishes the inductive construction.

To simplify notation, let us denote
$$Q=\omega^{<\omega}.$$
For every $s\in Q$, fix a continuous map $f_s:\beta X\to [0,1]$ witnessing $K_s$ being zero-subset of $\beta X$, i.e.
\begin{equation}
f^{-1}_s(0)=K_s. \label{(1)}
\end{equation}
Let $f:\beta X\to [0,1]^Q$ be the diagonal map given by the family $\{f_s:s\in Q\}$, i.e.
$$f(x)=\left(f_s(x)\right)_{s\in Q}\in [0,1]^Q.$$
Let $M=f(X)$ and $bM=f(\beta X)$. Since, $Q$ is countable, the space $bM$ is metrizable.

Note that

\begin{equation}
f(K_s)\subseteq bM\setminus M, \mbox{ for every } s\in Q.
\end{equation}
True, take $a\in f(K_s)$. We have $f(K_s)=\{y\in bM\subseteq [0,1]^Q:y(s)=0\}$, by \eqref{(1)}, so $a(s)=0$. On the other hand, if $a\in M$, then $a=f(x)$ for some $x\in X$ (because $M=f(X)$). But $K_s\cap X=\emptyset$, so $f(x)(s)=a(s)\neq 0$, by \eqref{(1)}; a contradiction.

Using the families $\{K_s:s\in Q\}$ and $\{V_s:s\in Q\}$, we will recursively construct a strategy $\tau$ for player I in the $k$-Porada game $kP(bM,bM\setminus M)$. We will make sure that if $(W_0,\ldots ,W_{m-1})$ is an admissible $m$-tuple for the strategy $\tau$ and if $\tau(W_0,\ldots ,W_{m-1})=(C_m,U_m)$, then there is $s\in Q$ of length $m$, such that

\begin{enumerate}[(i)]
 \setcounter{enumi}{5}
 \item $C_m=f(K_s)$
 \item $V_{s|i}\subseteq f^{-1}(W_{i-1})$, for every $1 \leq i \leq m$.
\end{enumerate}

Let $\tau(\emptyset)=(f(K_\emptyset),bM)$ and put $C_0=f(K_\emptyset)$, $U_0=bM$. Fix $n\geq 0$ and suppose that $\tau$ is defined for all admissible $n$-tuples $(W_0,\ldots W_{n-1})$ (if $n=0$, then a tuple $(W_0,\ldots W_{n-1})$ is empty) in such a way that conditions (vi)-(vii) hold. Suppose that
\begin{equation}
(W_0,\ldots ,W_{n}) \mbox{ is an admissible $(n+1)$-tuple.} \label{(3)}
\end{equation}
 By our inductive assumption, the pair $(C_n,U_n)=\tau(W_0,\ldots , W_{n-1})$ is well defined and (vi) holds, i.e. $C_n=f(K_s)$ for some $s\in \omega^n$. By \eqref{(3)}, we have
 $$f(K_s)=C_n\subseteq W_n \subseteq U_n.$$
 Thus,
 \begin{equation}
    K_s\subseteq f^{-1}(W_n). \label{(4)}
 \end{equation}
According to (ii), there is an open set $U_s\supseteq K_s$ such that
 $$(K_s,U_s)=\sigma(V_{s|0},V_{s|1}, \ldots , V_{s}).$$
Using (iii), (iv) and \eqref{(4)}, by compactness we can find $k\in\omega$ such that
\begin{equation}\label{zawieranie}
 V_{s\frown k}\subseteq f^{-1}(W_n)\cap U_s
\end{equation}
We set
$$\tau(W_0,\ldots , W_n)=(f(K_{s\frown k}), W_n).$$
This finishes the construction of $\tau$. Note that $\eqref{zawieranie}$ guarantees that (vii) holds.

By our assumption, the space $X$ is projectively Menger so $M$ must be Menger being a separable metrizable continuous image of $X$. Hence, by Theorem \ref{thrm_Telgarsky}, the strategy $\tau$ is not winning. It follows that there exists a play

$$\tau(\emptyset),\; W_0,\; \tau(W_0),\; W_1,\;\tau(W_0,W_1),\;\ldots$$

in which player I applies her strategy and fails, i.e.

\begin{equation}
    \emptyset\neq\bigcap_{n\in \omega}W_n\subseteq bM\setminus M. \label{(5)}
\end{equation}

   By condition (vii), the above play generates an infinite sequence $s\in \omega^\omega$ such that
   \begin{equation}
       V_{s|n}\subseteq f^{-1}(W_{n-1}), \mbox{ for every } n\in \omega. \label{(6)}
   \end{equation}
    According to (i) and (ii), for every $n$, the tuple
    $(V_{s|0},\ldots , V_{s|n})$ is admissible and $K_{s|n}$ is the compact set in the pair $\sigma(V_{s|0},\ldots , V_{s|n})$. This means that
    $$\sigma(\emptyset),\; V_{s|0},\; \sigma(V_{s|0}),\; V_{s|1},\; \sigma(V_{s|0}, V_{s|1}),\; \ldots$$
    is a play in the game $zP(\beta X, \beta X\setminus X)$. We claim that player II wins this run of the game.
    
    Indeed, we have $\bigcap_{n\in \omega}V_{s|n}=\bigcap_{n\in \omega}\cl_{\beta X}(V_{s|n})$, by (v). So this intersection must be nonempty by compactness. Moreover, $\bigcap_{n\in \omega}V_{s|n}\subseteq \beta X\setminus X$, by \eqref{(5)} and \eqref{(6)} and the fact that $f(X)=M$. This finishes the proof of $(A)\Rightarrow (B)$.
 

To prove the converse, assume that player I has no winning strategy in the game $zP(\beta X, \beta X\setminus X)$. Striving for a contradiction, suppose that $X$ is not projectively Menger. Consequently, $X$ maps continuously onto a separable metric non-Menger space $M$.
Let $f:X\to M$ be a continuous surjection. Since $M$ is a separable metric space, it has a metric compactification $bM$. Let
$\widetilde{f}:\beta X\to bM$ be the continuous extension of $f$. The space $M$ is not Menger, so by Theorem \ref{thrm_Telgarsky}, player I has a winning strategy $\sigma$ in the $k$-Porada game $kP(bM,bM\setminus M)$.

Using $\sigma$, we will inductively construct a winning strategy $\tau$ for player I in the game $zP(\beta X,\beta X\setminus X)$ in such a way that
for every $n\geq 0$, the following condition is satisfied:

\begin{enumerate}[(i)]
 \setcounter{enumi}{7}
\item If $(V'_0,\ldots ,V'_{n-1})$ is an admissible $n$-tuple in the game $zP(\beta X,\beta X\setminus X)$ (i.e. for the strategy $\tau$), then there is an admissible $n$-tuple $(V_0,\ldots ,V_{n-1})$ in the game $kP(bM,bM\setminus M)$ (i.e. for the strategy $\sigma$). And if
$\sigma(V_0,\ldots ,V_{n-1})=(K_n,U_n)$, then $\tau(V'_0,\ldots ,V'_{n-1})=\left(\widetilde{f}^{-1}(K_n), \widetilde{f}^{-1}(U_n)\right)$.
\end{enumerate} 

Let
$\sigma(\emptyset)=(K_0,U_0)$. Denote
$$K'_0=\widetilde{f}^{-1}(K_0)\quad \mbox{and}\quad U'_0=\widetilde{f}^{-1}(U_0).$$
Since $\widetilde{f}$ is surjective, the sets $K'_0$ and $U'_0$ are nonempty. Moreover, $K'_0$ being closed $G_\delta$-subset of $\beta X$, is a zero-set in $\beta X$.
In addition, $K'_0\subseteq \beta X\setminus X$. This is because $\widetilde{f}(X)=f(X)=M$ and $K_0$ misses $M$. Hence, the pair
$$\tau(\emptyset)=(K'_0,U'_0)$$
is a legal move for player I in the game $zP(\beta X, \beta X\setminus X)$ and the condition (viii) holds for $n=0$.

Fix $n\geq 0$ and suppose that $\tau$ is defined for all admissible $m$-tuples $(V_0,\ldots , V_{m-1})$, where $m\leq n$, and the condition (viii) holds for all such tuples. Let
$(V'_0,\ldots , V'_n)$ be an arbitrary admissible $(n+1)$-tuple in $zP(\beta X,\beta X\setminus X)$ and let $\tau(V'_0,\ldots ,V'_{n-1})$
be the last move for player I (it is well defined by the inductive assumption).
According to (viii), there is an admissible $n$-tuple $(V_0,\ldots , V_{n-1})$ in the game $kP(bM,bM\setminus M)$ such that
if $\sigma(V_0,\ldots , V_{n-1})=(K_n,U_n)$, then
$$\tau(V'_0,\ldots , V'_{n-1})=\left(\widetilde{f}^{-1}(K_n), \widetilde{f}^{-1}(U_n)\right).$$
The $(n+1)$-tuple $(V'_0,\ldots , V'_n)$ is admissible which means that
$V'_n$ is an open subset of $\beta X$ with
$$\widetilde{f}^{-1}(K_n)\subseteq V'_n\subseteq \widetilde{f}^{-1}(U_n).$$
Let
$$V_n=\widetilde{f}^\#(V'_n).$$
By Lemma \ref{obraz_kreteczka}, the set $V_n$ is open in $bM$ and $K_n\subseteq V_n \subseteq U_n$. This means that
the tuple $(V_0,\ldots, V_{n-1},V_n)$ is admissible in the game $kP(bM, bM\setminus M)$. Let
$$\sigma(V_0,\ldots, V_{n-1},V_n)=(K_{n+1},U_{n+1}).$$
Denote
$$K'_{n+1}=\widetilde{f}^{-1}(K_{n+1})\quad \mbox{and}\quad U'_{n+1}=\widetilde{f}^{-1}(U_{n+1}).$$
As for $K'_0$ and $U'_0$ we argue that these sets are nonempty, $K'_{n+1}$ is a zero-set in $\beta X$, $U'_{n+1}$ is open in $\beta X$, and
$K'_{n+1}\subseteq \beta X\setminus X$. Since $U_{n+1}\subseteq V_n=\widetilde{f}^\#(V'_n)$ we must have
$U'_{n+1}\subseteq V'_n$ (cf. Lemma \ref{obraz_kreteczka}). Therefore, the pair
$$\tau(V'_0,\ldots , V'_n)=(K'_{n+1}, U'_{n+1})$$
is a well-defined $(n+1)$-st move for player I in the game $zP(\beta X, \beta X\setminus X)$. This finishes the induction.

Let us prove that the strategy $\tau$ that we have just constructed, is winning for player I. To this end, consider an arbitrary play $P$:
$$\tau(\emptyset),\; V'_0,\; \tau(V'_0),\; V'_1,\;\tau(V'_0,\;V'_1),\; \ldots$$
in the game $zP(\beta X,\beta X\setminus X)$, where player I applies the strategy $\tau$. We need to show that either
$$\bigcap_{n\in \omega} V'_n=\emptyset \quad \mbox{or}\quad X\cap \bigcap_{n\in \omega} V'_n\neq \emptyset.$$
By (viii), the play $P$ induces a play
$$\sigma(\emptyset),\; V_0,\; \sigma(V_0),\; V_1,\;\sigma(V_0,\;V_1),\; \ldots$$
in the game $kP(bM, bM\setminus M)$ such that, for every $n\in \omega$ we have:
\begin{align}
&\mbox{If } \sigma(V_0,\ldots ,V_{n-1})=(K_n,U_n), \mbox{ then } \label{(7)}\\ 
&\tau(V'_0,\ldots ,V'_{n-1})=\left(\widetilde{f}^{-1}(K_n), \widetilde{f}^{-1}(U_n)\right).\nonumber
\end{align}
By our assumption, the strategy $\sigma$ is winning for player I, so either
\begin{equation}
    \bigcap_{n\in \omega} V_n=\bigcap_{n\in \omega} U_n=\emptyset\label{(8)}
\end{equation}
or
\begin{equation}
    M\cap\bigcap_{n\in \omega} V_n=M\cap\bigcap_{n\in \omega} U_n\neq \emptyset. \label{(9)}
\end{equation}
According to \eqref{(7)}, we have
$$\bigcap_{n\in\omega}V'_n=\bigcap_{n\in\omega} \widetilde{f}^{-1}(U_n).$$
So, if \eqref{(8)} holds, then $\bigcap_{n\in\omega}V'_n=\emptyset$, which means that player I wins in the game $zP(\beta X,\beta X\setminus X)$.
Suppose that \eqref{(9)} holds and pick $y\in M\cap\bigcap_{n\in \omega} U_n$. Since the map $f:X\to M$ is surjective, there is $x\in X$ with
$f(x)=\widetilde{f}(x)=y$, whence
$$x\in X\cap\bigcap_{n\in\omega} \widetilde{f}^{-1}(U_n)= X\cap \bigcap_{n\in \omega} V'_n.$$
In particular the latter set is nonempty which means that player I wins in this case too.
\end{proof}

\begin{rem}
It is known that a space $X$ has the Menger property if and only if player I has no winning strategy in the so-called Menger game $M(X)$ (see \cite[Section 2.3]{AD}).
Actually Telg\'arsky proved in \cite{T} a result more general than Theorem \ref{thrm_Telgarsky} stated above. Namely, he showed that player $P$ has a winning strategy in the $k$-Porada game $kP(bX,bX\setminus X)$ if and only if player $P$ has a winning strategy in the Menger game $M(X)$ on $X$, where $P\in \{I,II\}$ and $bX$ is a compactification of $X$. Now, our Theorem \ref{projectively_Menger_characterisation} can be reformulated in the following way: Player I has a winning strategy in the $z$-Porada game $zP(\beta X, \beta X\setminus X)$ if and only if, there is a continuous map $f:X\to Y$ of $X$ onto a separable metric space $Y$ such that player I has a winning strategy in the Menger game $M(Y)$.

We do not know if a similar equivalence holds true for player II.
\end{rem}

\section*{Acknowledgements}
The results presented in the paper were obtained as part of the second author's master's thesis written under the supervision of the first author at the University of Warsaw.

The research of the first author was supported by
Fundaci\'{o}n S\'{e}neca - ACyT Regi\'{o}n de Murcia project 21955/PI/22, Agencia Estatal de Investigación (Government of Spain), project PID2021-122126NB-C32 funded by\\MCIN/AEI /10.13039/501100011033 / FEDER, EU and NextGenerationEU funds through Mar\'{i}a Zambrano fellowship.

\bibliographystyle{siam}

\end{document}